\documentclass[12pt,a4paper,oneside]{article}
\usepackage[inner=3cm,outer=3 cm,top=2.5cm,bottom=2.5cm]{geometry}
\usepackage{amssymb}
\usepackage{amsthm} 
\usepackage{amsmath}
\usepackage{graphicx}
\usepackage{xcolor}
\usepackage{t1enc}
\usepackage[utf8]{inputenc}
\usepackage{enumerate}
\usepackage[UKenglish]{babel}
\usepackage{float}
\usepackage[unicode]{hyperref}
\usepackage{epstopdf}
\usepackage{indentfirst}
\usepackage{enumitem}
\setlist{leftmargin=2cm}
\usepackage{booktabs}
\usepackage{upgreek}
\usepackage{bm}
\usepackage{graphicx}
\usepackage{csquotes}
\usepackage{authblk}

\DeclareMathAlphabet{\mymathbb}{U}{bbold}{m}{n}

\hypersetup{
	colorlinks=false,
	linkcolor=red,
	urlcolor=blue,
	citecolor=gray
}						
\newtheorem{dfn}{Definition}[section]
\newtheorem{thm}[dfn]{Theorem}
\newtheorem{claim}[dfn]{Claim}

\newtheorem{corollary}[dfn]{Corollary}
\newtheorem{conjecture}[dfn]{Conjecture}
\newcommand{\prf}{\noindent{\textit{Proof:}}}
\newcommand{\clm}{\noindent{\textit {Claim:}}}

\begin{document}

\title{Coloring one-headed directed hypergraphs }
\author{Balázs István Szabó\thanks{Research supported by the EXCELLENCE-24 project no.~151504 Combinatorics and Geometry of the NRDI Fund.}}
\affil{ELTE E\"{o}tv\"{o}s Lor\'{a}nd University, Budapest}

\maketitle
\begin{abstract}{ A directed hypergraph is a hypergraph in which the vertex set of each hyperedge is partitioned into two disjoint parts, a head and a tail. Keszegh and Pálvölgyi posed the following conjecture. Let $H$ be a directed hypergraph such that in every hyperedge the number of head-vertices is less than the number of tail-vertices and assume that for every pair of hyperedges $e_{1},e_{2}\in E(H)$ with $|e_{1}\cap e_{2}|=1$, the common vertex is a head-vertex in at least one of the hyperedges. Then $H$ admits a proper 2-coloring. Keszegh showed that the conjecture is also true in the special case of 3-uniform hypergraphs \cite{2}. A directed hypergraph is called one-headed if every hyperedge has exactly one head-vertex. The main result of this paper is that the conjecture is true for one-headed directed hypergraphs with all hyperedges having size at least three.
	
Directed 3-uniform hypergraphs such that in every hyperedge the number of head-vertices is one and the number of tail-vertices is two are called $2\rightarrow 1$ hypergraphs. In this paper we consider sufficient conditions for $2\rightarrow 1$ hypergraphs to be proper $k$-colorable for some small $k$. }\end{abstract}

\section{Introduction}

A hypergraph is a pair $(V,E)$ where $V$ is a finite set and $E$ is a family of non-empty subsets of $V$. The elements of the set $V$ are called vertices, and the elements of the set $E$ are called hyperedges. A hypergraph is $\emph{properly k-colorable}$ if its vertices can be colored with $k$ colors such that each edge contains vertices from at least two different color classes and such a coloring is called a proper $k$-coloring. The $\emph{chromatic number}$ of a hypergraph is $k$ if it admits a proper $k$-coloring but not a proper $(k-1)$-coloring.

Lovász proved the following result which gives a sufficient condition for proper 2-colorability.
\begin{thm}\label{one_intersection}\cite{5}
	Let $H$ be a hypergraph in which every pair of hyperedges has an empty intersection or intersects in at least two vertices. Then $H$ admits a proper 2-coloring.
\end{thm}
We get a generalization of this theorem if we do not require for every pair of hyperedges to have an empty intersection or intersection with size at least two. We need the definition of directed hypergraphs.

\begin{dfn}{
		A directed hypergraph is a hypergraph in which the vertex set of each hyperedge is partitioned into two disjoint parts, a head and a tail. Accordingly, the vertices in the head are called head-vertices and the vertices in the tail are called tail-vertices.
		For a hyperedge $e$, $h(e)$ denotes the set of head-vertices of the hyperedge $e$ and $t(e)$ denotes the set of its tail-vertices.
	}
\end{dfn}

In a directed hypergraph if two hyperedges intersect in a vertex $v$, then there are three different cases: $v$ is head-vertex of both hyperedges, head-vertex one of them and tail-vertex of the other or tail-vertex of both.

 Keszegh and Pálvölgyi have the following conjecture which generalizes Theorem \ref{one_intersection}.
\begin{conjecture}\label{conjecture}\cite{2}
	Let $H$ be a directed hypergraph such that each hyperedge has more tail-vertices than head-vertices, and suppose that for every $e_{1},e_{2}\in E(H)$ with $|e_{1}\cap e_{2}|=1$, the common vertex is the head-vertex of at least one of $e_{1},e_{2}$. Then there exists a proper 2-coloring of the hypergraph $H$.
\end{conjecture}

  Keszegh showed that the conjecture is true for 3-uniform hypergraphs \cite{2}.

\begin{thm}\label{H_{1}}\cite{2}
	Let $H$ be a $3$-uniform hypergraph in which every hyperedge has two tail-vertices and one head-vertex. Suppose that if $e_{1},e_{2}\in E(H)$ with $|e_{1}\cap e_{2}|=1$, then the common vertex is a head-vertex in at least one of $e_{1}$ and $e_{2}$. Then $H$ admits a proper 2-coloring.
\end{thm}

Further generalizing this, our main result is that the conjecture is also true for hypergraphs in which every hyperedge has exactly one head-vertex.
 
 \begin{thm}\label{one_head}{
 		Let $H$ be a directed hypergraph such that in every hyperedge the number of tail-vertices is at least two and the number of head-vertices is exactly one, and suppose that for every $e_{1},e_{2}\in E(H)$ with $|e_{1}\cap e_{2}|=1$, the common vertex $v$ is the head-vertex of at least one of $e_{1},e_{2}$. Then $H$ admits a proper 2-coloring.}
 \end{thm}
Note that in Theorem \ref{one_head} $H$ does not need to be uniform.
The 3-uniform directed hypergraphs in which every hyperedge has exactly one head-vertex are called $2\rightarrow 1$ hypergraphs. A hyperedge can be written as $ab\rightarrow c$, where $a$ and $b$ are the two tail-vertices of the hyperedge and $c$ is its head-vertex. Cameron was interested in the extremal properties of $2\rightarrow 1$ hypergraphs avoiding a $2\rightarrow 1$ hypergraph with two edges \cite{1}. The $2\rightarrow 1$ hypergraphs studied by Cameron are the following.

\begin{eqnarray*}
	V(H_{2})&=&\{a,b,c,d\}, \  \ E(H_{2})=\{ab\rightarrow c,ab\rightarrow d\}\\
	V(I_{1})&=&\{a,b,c,d\}, \  \ E(I_{1})=\{ab\rightarrow c,ad\rightarrow c\}\\
	V(R_{3})&=&\{a,b,c,d\}, \  \ E(R_{3})=\{ab\rightarrow c,bc\rightarrow d\}\\
	V(E)&=&\{a,b,c,d\}, \  \ E(E)=\{ab\rightarrow c,dc\rightarrow b\}\\
	V(I_{0})&=&\{a,b,c,d,e\}, \  \ E(I_{0})=\{ab\rightarrow e,cd\rightarrow e\}\\
	V(H_{1})&=&\{a,b,c,d,e\}, \  \ E(H_{1})=\{ab\rightarrow c,ad\rightarrow e\}\\
	V(R_{4})&=&\{a,b,c,d,e\}, \  \ E(R_{4})=\{ab\rightarrow c,cd\rightarrow e\}
\end{eqnarray*}

Given a pair of directed hypergraphs $H$ and $G$, we say that $G$ is a $\emph{subhypergraph}$ of $H$ if one can obtain $G$ from $H$ by removing edges and vertices.
Theorem \ref{H_{1}} can also be stated as follows: if a $2\rightarrow 1$ hypergraph does not contain the hypergraph $H_{1}$ as a subhypergraph, then it admits a proper 2-coloring. We study the chromatic number of $2\rightarrow 1$ hypergraphs that avoid any one of the above two-edge $2\rightarrow 1$ hypergraphs. 

\begin{claim}\label{H_{2}}
	For every integer $k\geq 2$, there exists a $2\rightarrow 1$ hypergraph $H=(V,E)$ such that it does not contain $H_{2}$ as a subhypegraph, namely if $e_{1},e_{2}\in E$ with $e_{1}\cap e_{2}=\{u,v\}$, then $u$ or $v$ is a head-vertex of at least one hyperedge and the chromatic number of $H$ is at least k.
\end{claim}

For the hypergraphs $I_{1}$, $R_{3}$ and $E$, we can also give $2 \rightarrow 1$ hypergraphs with chromatic number at least k such that it does not contain these hypergraphs as subhypegraphs. Specifically, a $2\rightarrow 1$ hypergraph with chromatic number at least $k$ and not containing any of the hypergraphs $I_{1}$, $R_{3}$ and $E$ as subhypergraphs can be given. It is easy to check that a $2\rightarrow1$ hypergraph does not contain any of the hypergraphs $I_{1}$, $R_{3}$ and $E$ as subhypergraphs if and only if it does not contain two hyperedges intersecting in exactly two vertices which are not tail-vertices in both of them.

\begin{claim}\label{I_{1},R_{3},E}
	Let $k\geq 2$ be an integer. Then there exists a $2\rightarrow 1$ hypergraph $H$ with chromatic number at least k and not containing any of the hypergraphs $I_{1}$, $R_{3}$ and $E$ as subhypergraphs, namely if $e_{1},e_{2}\in E(H)$ with $e_{1}\cap e_{2}=\{u,v\}$, then $t(e_{1})=t(e_{2})=\{u,v\}$.
\end{claim}
On the other hand avoiding the hypergraph $I_{0}$ is a sufficient condition for proper 4-colorability.
\begin{thm}\label{I_{0}}
	Let $H$ be a $2\rightarrow 1$ hypergraph. Suppose that $H$ does not contain $I_{0}$ as a subhypergraph, namely if $e_{1},e_{2}\in E(H)$ with $|e_{1}\cap e_{2}|=1$, then the common vertex is a tail-vertex of at least one hyperedge. Then $H$ admits a proper 4-coloring.
\end{thm}
Note that three colors might be needed and there is a gap.
\begin{thm}\label{R_{4}}
	Let $H$ be a $2\rightarrow 1$ hypergraph. Suppose that $H$ does not contain the hypergraph $R_{4}$ as subhypergraph, namely if $e_{1},e_{2}\in E(H)$ and $|e_{1}\cap e_{2}|=1$, then the common vertex is either the head-vertex of both hyperedges or a tail-vertex of both hyperedges. Then there exists a proper 3-coloring of the hypergraph $H$.
\end{thm}
Note that this is best. Theorem \ref{R_{4}} can be further generalized, see later Theorem \ref{general_head_tail}.
The new results and the previous result about $H_{1}$ are contained in Table \ref{table2}, where $\chi(H)$ denotes the chromatic number of the hypergraph $H$.\newline
\begin{table}[h!!]
	\centering
	
	\begin{tabular}{|l|cc|}
		\hline
		$F$ & $\sup\{\chi(H):$&$H$ avoids $F$ $\}$ \\\hline
		$H_{2}$ & $\infty$&(Claim \ref{H_{2}} )\\\hline
		$I_{1}$ & $\infty$&(Claim \ref{I_{1},R_{3},E})\\\hline
		$R_{3}$ & $\infty$&(Claim \ref{I_{1},R_{3},E})\\\hline
		$E$ & $\infty$&(Claim \ref{I_{1},R_{3},E} )\\\hline
		$I_{0}$ & $3\leq,\leq 4$&(Theorem \ref{I_{0}})\\\hline
		$H_{1}$ & 2&(Theorem \ref{one_head})\\\hline
		$R_{4}$ & 3&(Theorem \ref{R_{4}})\\\hline
	\end{tabular}
	\caption{}
	\label{table2}
\end{table}

For more about how these problems are related to other problems see \cite{2}, where relations to \cite{3,4,6} are detailed.

It is also interesting to consider the case when we avoid a hypergraph with more than one hypergraph. A $\emph{good coloring}$ of a $3$-uniform hypergraph $H$ is the following \cite{7}.
Consider the graph $G_{H}$ whose vertex set is $V(H)$ and its edges are those vertex pairs $u,v\in V(H)$ for which there exists a hyperedge $e\in E(H)$ such that $\{u,v\}\subset e$. A good coloring of the hypergraph $H$ is an edge-coloring of the graph $G_{H}$ with colors red and blue and a direction of the red edges such that every hyperedge $e=\{u,v,w\}\in E(H)$, the edges $uv$ and $uw$ are red, the edge $vw$ is blue in the graph $G_{h}$ and the two red hyperedges are directed to $u\rightarrow v$ and $u\rightarrow w$ for some ordering of $u,v,w$. 
\begin{thm}\cite{7}{
	Let $H$ be a 3-uniform hypergraph on $n$ vertices. If $H$ admits a good coloring then the number of hyperedges of $H$ is at most $f(n)$, where $f(0)=1$ and $f(n)=\max_{k\in [n-1]} {k \choose 2}\left(n-k\right)+f(n-k)$.}
\end{thm}
It is easy to check that a 3-uniform hypergraph $H$ has a good coloring if and only if can be given a $2\rightarrow 1$ direction of each hyperedge of $H$ such that the resulting $2\rightarrow 1$ hypergraph avoids the hypergraphs $R_{3}$ and $E$. Thus, if a $2\rightarrow 1$ hypergraph $H$ on $n$ vertices does not contain any of the hypergraphs $R_{3}$ and $E$ as subhypergraphs, then $|E(H)|\leq f(n)$.
By Claim \ref{I_{1},R_{3},E} for such families no proper-colorability result is possible. On the other hand, while avoiding hypergraphs $I_{0}$ and $R_{4}$ separately is not, but avoiding both hypergraphs is a sufficient condition for proper 2-colorability.
\begin{thm}\label{I_{0},R_{4}}
	{
		Let $H$ be a $2\rightarrow 1$ hypergraph and suppose that $H$ avoids both $I_{0}$ and $R_{4}$, namely if $e_{1},e_{2}\in E(H)$ with $|e_{1}\cap e_{2}|=1$, then the common vertex is a tail-vertex of both hyperedges. Then $H$ admits a proper 2-coloring.
	}
\end{thm}

\section{Proof of Theorem \ref{one_head}}
We show that Conjecture \ref{conjecture} is true in a special case, namely it is sufficient to assume that every hyperedge has exactly one head-vertex. The condition of the conjecture that every hyperedge has more tail-vertices than head-vertices is equivalent in this case to that every hyperedge has at least two tail-vertices.
Let $V$ be the set of vertices of the hypergraph $H$ and assume that $|V|=n$. In the following, denote the head-vertex of a hyperedge $e$ by $h(e)$ and the set of its tail-vertices by $t(e)$. If $e_{1}$ and $e_{2}$ are two hyperedges for which $e_{1}\subseteq e_{2}$, then we can drop the hyperedge $e_{2}$, since if the hyperedge $e_{1}$ is non-monochromatic then $e_{2}$ is not monochromatic. So we can assume that neither of the hyperedges contains the other.\newline
Take two colors, red and blue. At the beginning let all the vertices be blue. 
Before the i-th step we have processed the vertices $v_1,...v_{i-1}$, in the i-th step either we have to process the single vertex already fixed as $v_{i}$ in the previous step, or we can choose an arbitrary not yet processed vertex as $v_{i}$. 
Let $v_{1}$ be an arbitrary vertex of $H$.
Starting from the vertex $v_{1}$, perform the following two steps for each vertex $v_{i}$ one-by-one:\newline
	Step i: We check if there exists a hyperedge $e$ which is monochromatic blue, $v_{i}$ is the tail-vertex of the hyperedge $e$ and $t(e)\subseteq \{v_{1},v_{2},...,v_{i}\}$. 
	If there is no such hyperedge, then we arbitrarily choose a vertex $v_{i+1}\in V\setminus \{v_{1},v_{2},...,v_{i}\}$ and proceed to Step $i$.
	Otherwise if such a hyperedge exists, then we color the vertex $v_{i}$ red and we check whether the coloring of the vertex $v_{i}$ to red creates a monochromatic red hyperedge, if so we color the vertex $v_{i-1}$ blue. If $e$ can be chosen such that $h(e)\in V\setminus \{v_{1},v_{2},...,v_{i}\}$ is satisfied, then the point $h(e)$ is chosen as vertex $v_{i+1}$. Otherwise $v_{i+1}$ is chosen arbitrarily from the set $V\setminus \{v_{1},v_{2},...,v_{i}\}$, provided $i\leq n-1$.
	
	We are left to prove that this simple coloring algorithm gives a proper 2-coloring.
	Notice that the index of a vertex cannot change after it has been processed, and the color of the vertex $v_{j}$ can only change to red during its processing and can only change to blue during processing of the vertex $v_{j+1}$, so the color of a vertex can change at most twice. The set of the hyperedges can be divided into two parts according to whether the vertex of the hyperedge with the highest index is a tail-vertex or a head-vertex of the hyperedge. Let $E_{t}$ be the set of hyperedges in which the vertex with the highest index is a tail-vertex of the hyperedge and let $E_{h}$ be the set of hyperedges in which the vertex with the highest index is a head-vertex of the hyperedge. Let the tail-vertex of a hyperedge $e$ with the largest index denoted by $t_{max}(e)$.
	\begin{claim}\label{monochromatic_red}
		
		During the coloring algorithm a monochromatic red hyperedge whose vertex with the largest index is a tail-vertex of the hyperedge cannot be created.
	\end{claim}

	\begin{proof}{
		Suppose for a contradiction that there exists such a hyperedge $g$, all vertices of $g$ are colored red and the vertex of hyperedge $g$ with the highest index is a tail-vertex of $g$. Then, the hyperedge $g$ is made monochromatic red by coloring the vertex $t_{max}(g)$ red. Since we have colored the vertex $t_{max}(g)$ red, there is a hyperedge $f$ which was monochromatic blue and $t_{max}(f)=t_{max}(g)$. In this case, $g$ and $f$ intersect in one vertex, since all vertices of $f$ except the vertex $t_{max}(f)$ were blue and all vertices of $g$ were red when $t_{max}(f)$ is colored red, so $f$ and $g$ have exactly one common vertex, $t_{max}(f)$, which is a tail-vertex of both hyperedges. This is a contradiction, so during the algorithm, it is never possible for a hyperedge from the set $E_{t}$ to become monochromatic red.}\end{proof}
	
		Claim \ref{monochromatic_red} in other words says that in the coloring algorithm, if a monochromatic red hyperedge is created, then its vertex with the largest index can only be the head-vertex of the hyperedge. In the following we rule out this case too.

	\begin{claim}\label{consecutive}{
			If all the tail-vertices of a hyperedge $g$ are colored red during the algorithm and the vertex with the highest index is the head-vertex of $g$, then the vertices of $g$ are consecutive vertices.
		}
	\end{claim}
	\begin{proof}{ Let $t_{g}$ denote the number of tail-vertices of the hyperedge $g$, from the condition of the theorem we know that $t_{g}\geq 2$.
		We prove that for each $k\leq t_{g}$, the vertices of the hyperedge $g$ with the $k+1$ smallest indices are consecutive. We prove this statement by induction. First, consider the case $k=1$.\newline
		Since $g\in E_{h}$, we processed the head-vertex of $g$ last. Let $t_{j}(g)$ denote the tail-vertex of the hyperedge $g$ with the $j$-th smallest index. Since the color of a vertex can change from blue to red at most once and all the vertices of $g$ will be red, we color the vertex $t_{1}(g)$ red when it is processed. Hence there exists a hyperedge $e_{1}$ such that $t_{max}(e_{1})=t_{1}(g)$ and $e_{1}$ was monochromatic blue before processing the vertex $t_{1}(g)$. If the index of $h(e_{1})$ is smaller than the index of $t_{max}(e_{1})$, then the only common vertex of $e_{1}$ and $g$ is $t_{1}(g)$, which is the tail-vertex of both, this is a contradiction. So $h(e_{1})$ has a larger index than $t_{max}(e_{1})$. Hence there exists also a hyperedge $e_{1}'$ for which $t_{max}(e_{1}')=t_{1}(g)$, $h(e_{1}')$ is the neighbor of the vertex $t_{max}(e_{1}')$ with higher index and $e_{1}'$ was monochromatic blue before $t_{1}(g)$ was processed. In this case, $h(e_{1}')$ is also a vertex of the hyperedge $g$, because otherwise the only common vertex of $e_{1}'$ and $g$ would be $t_{1}(g)$, which is the tail-vertex of both, which would be a contradiction. It follows from these that $t_{2}(g)=h(e_{1}')$, so the two tail-vertices of $g$ with the lowest indices are consecutive vertices. \newline 
		The case $k>1$ is handled in a similar way.
		Let $j < t_{g}$ and suppose that for $k=j$ the statement is true, so the $j+1$ vertices of the hyperedge $g$ with the smallest indices are consecutive vertices. Since $j<t_{g}$ and $g\in E_{h}$, the $(j+1)$-th processed vertex of $g$ is a tail-vertex of $g$. All the vertices of $g$ can be red if and only if at that moment the vertex $t_{j+1}(g)$ is colored red, the vertices of $g$ with the $j$ smallest indices are red. The condition for the vertex $t_{j+1}(g)$ to be red is that there exists an edge $e_{j+1}$ for which $t_{max}(e_{j+1})=t_{j+1}(g)$ and $e_{j+1}$ is monochromatic blue before the processing of the vertex $t_{j+1}(g)$. If the index of $h(e_{j+1})$ is smaller than the index of $t_{max}(e_{j+1})$, then the vertex of $e_{j+1}$ with the highest index is $t_{max}(e_{j+1})$. Before the processing the vertex $t_{j+1}(g)=t_{max}(e_{j+1})$ all vertices of $g$ with index less than $t_{j+1}(g)$ are red, and all vertices of $e_{j+1}$ are blue and have index no greater than $t_{j+1}(g)$, which implies that the only common vertex of $g$ and $e_{j+1}$ is $t_{j+1}(g)$. This vertex is also a tail-vertex of the hyperedge $e_{j+1}$ and $g$, which is a contradiction. 
		So $h(e_{j+1})$ has a larger index than $t_{max}(e_{j+1})$. Hence there exists a hyperedge $e_{j+1}'$ also such that $t_{max}(e_{j+1}')=t_{j+1}(g)$, $h(e_{j+1}')$ is the neighbor of the vertex $t_{max}(e_{j+1}')$ with the higher index and $e_{j+1}'$ was monochromatic blue before the vertex $t_{j+1}(g)$ was processed. If $h(e_{j+1}')$ is not a vertex of the hyperedge $g$, then the only common vertex of $e_{j+1}'$ and $g$ is $t_{j+1}(g)$, since before the processing of the vertex $t_{j+1}(g)$ the vertices of the hyperedge $g$ with index less than $t_{j+1}(g)$ were red, and the vertices of the hyperedge $e_{j+1}'$ with index less than $t_{j+1}(g)$ were blue and the hyperedge $e_{j+1}'$ has only one vertex with index greater than the index of $t_{j+1}(g)$. Then we get a contradiction, since $t_{j+1}(g)$ is a tail-vertex of both hyperedges. So $h(e_{j+1}')$ is a vertex of $g$, which is the neighbor of $t_{j+1}(g)$ with higher index. Using the inductive assumption, we obtain that the $j+2$ vertices of the hyperedge $g$ with the smallest indices are consecutive. \newline 
		So we proved that for every $k\leq t_{g}$ the $k+1$ vertices of $g$ with smallest indices are consecutive.
		In case $k=t_{g}$, we obtain that the vertices of $g$ are consecutive.
		}\end{proof}
	
	\begin{corollary}\label{unique}
		If the algorithm creates a monochromatic red hyperedge, then its vertices are consecutive vertices.
	\end{corollary}
	\begin{proof}{
		Suppose that the algorithm creates a monochromatic red hyperedge, denoted by $g$. It follows from Claim \ref{monochromatic_red} that $g\in E_{h}$. Since all the tail-vertices of $g$ are colored red and the vertex of $g$ with the highest index is the head-vertex of $g$ we can apply Claim \ref{consecutive}, which says that in this case the vertices of $g$ are consecutive.}\end{proof}

	Suppose that by coloring the vertex $v_{j}$ to red, the hyperedges $e_{1}$ and $e_{2}$ both become monochromatic red. From Claim \ref{monochromatic_red} we know that $e_{1},e_{2}\in E_{h}$. Then, $h(e_{1})=h(e_{2})=v_{j}$, and by Corollary \ref{unique}, $e_{1}$ and $e_{2}$ are also composed of consecutive vertices. It follows that $e_{1}\subseteq e_{2}$ or $e_{2}\subseteq e_{1}$. Since we have assumed that neither of the hyperedges contains the other, $e_{1}$ and $e_{2}$ denote the same hyperedge.
	
	\begin{corollary}\label{at_most_one}{
		During the algorithm coloring a vertex red can create at most one monochromatic red hyperedge.}
	\end{corollary}
	
	\begin{claim}{
		At the end of the algorithm, there cannot be monochromatic red hyperedges.}
	\end{claim}
	\begin{proof}{
		Suppose for a contradiction that $g$ is a monochromatic red hyperedge at the end of the algorithm. We know from Claim \ref{monochromatic_red} that the last vertex of $g$ is its head-vertex, and Corollary \ref{unique} says that the vertices of $g$ are consecutive. The hyperedge $g$ can be monochromatic only by coloring the vertex $h(g)$ red. By Corollary \ref{at_most_one}, at most one monochromatic red hyperedge can be created by coloring vertex $h(g)$ red, which is $g$.
	Under the processing of the vertex $h(g)$ we color the neighbor of $h(g)$ with lower index blue, which is also a vertex of the hyperedge $g$, since the vertices of $g$ are consecutive vertices and $g\in E_{h}$. From now on the neighbor of $h(g)$ with smaller index, $t_{max}(g)$ is left unchanged, so at the end of the algorithm it will be blue, which contradicts the fact that $g$ is a monochromatic red hyperedge at the end of the coloring. We proved that at the end of the algorithm there cannot be monochromatic red hyperedges. }\end{proof}
	
	\begin{claim}{
			At the end of the algorithm, there cannot be monochromatic blue hyperedges.
		}
	\end{claim}
	\begin{proof}{ 
		Since all hyperedges will have at least one red vertex during the algorithm, a monochromatic blue hyperedge can only be created by coloring a vertex blue due to a monochromatic red hyperedge $g$ and assume that $g$ has become monochromatic red by coloring a vertex $v$ red. By Corollary \ref{at_most_one}, we know that $g$ is unique and by Claim \ref{monochromatic_red} $g\in E_{h}$, so $h(g)=v$. At the same time, using Corollary \ref{unique}, we know that the vertices of $g$ are consecutive vertices. So we color the neighbor of $h(g)$ with smaller index, which is $t_{max}(g)$ blue due to the hyperedge $g$. 
		Suppose for a contradiction that coloring the vertex $t_{max}(g)$ blue creates a monochromatic blue hyperedge $d$ such that none of its vertices will be red in the future, so it will remain monochromatic blue until the end of the algorithm. \newline 
		If the index of $t_{max}(d)$ is greater than the index of $h(g)$, then after the processing of the vertex $h(g)$ there will be a vertex of $d$ which is colored red, so it will not remain monochromatic blue until the end of the algorithm, because if $d$ is monochromatic blue before processing of the vertex $t_{max}(d)$, then we must color the vertex $t_{max}(d)$ red. We know that $t_{max}(d)$ and $h(g)$ are different vertices because $t_{max}(d)$ is blue and $h(g)$ is red when $h(g)$ is processed.
		\newline 
		Thus we can assume that $t_{max}(d)$ has a smaller index than the vertex $h(g)$.
		If $t_{max}(g)$ is a tail-vertex of the hyperedge $d$, then $g$ and $d$ have exactly one common vertex $t_{max}(g)$, since after $t_{max}(g)$ is colored blue, all vertices of $d$ are blue and all vertices of $g$ are red except for $t_{max}(g)$. Then the vertex $t_{max}(g)$ is the only common vertex of the two hyperedges and it is a tail-vertex of both, which is a contradiction. Thus we can assume that $t_{max}(g)$ is the head-vertex of the hyperedge $d$. Then the index of $h(d)$ is larger than the index of $t_{max}(d)$ since $t_{max}(d)$ has smaller index than $h(g)$ and $h(d)=t_{max}(g)$ is the neighbor of the vertex $h(g)$ with smaller index.
		Using the fact that every hyperedge has at least two tail-vertices, we denote the vertex of hyperedge $g$ with the second highest index by $t_{max-1}(g)$.
		The vertex $t_{max}(d)$ cannot be a neighbor of $h(d)$, since the neighbor of $h(d)=t_{max}(g)$ with smaller index is $t_{max-1}(g)$, which is red when $t_{max}(g)$ is colored blue, as shown in Figure \ref{no_blue_1}.
		\begin{figure}[h!]
			\centering
			\includegraphics[width=0.7\textwidth]{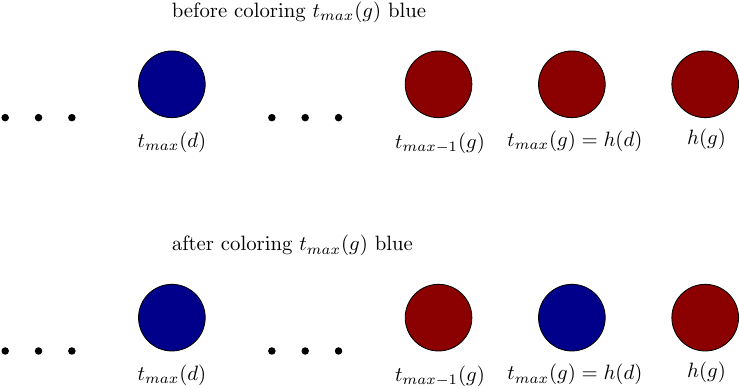}
			\caption{}
			\label{no_blue_1}
			
		\end{figure}
		
		We know that when $h(d)=t_{max}(g)$ is colored blue, $t_{max}(d)$ was blue.
		This is only possible if either we did not color $t_{max}(d)$ red at all, or we colored it red but colored it back to blue due to a monochromatic red hyperedge $e'\in E_{h}$ for which $t_{max}(e')=t_{max}(d)$.
		If we have not colored $t_{max}(d)$ red once, then there is a vertex in the hyperedge $d$ with index less than $t_{max}(d)$ which was red before processing of $t_{max}(d)$ and it must have remained red. This contradicts the fact that by coloring $h(d)$ blue, $d$ became monochromatic blue. So there is only one possibility, we colored the vertex $t_{max}(d)$ red and then colored it back to blue due to a monochromatic red hyperedge $e'\in E_{h}$ as illustrated in Figure \ref{no_blue_2}.
		\begin{figure}[h!]
			\centering
			\includegraphics[width=0.37\textwidth]{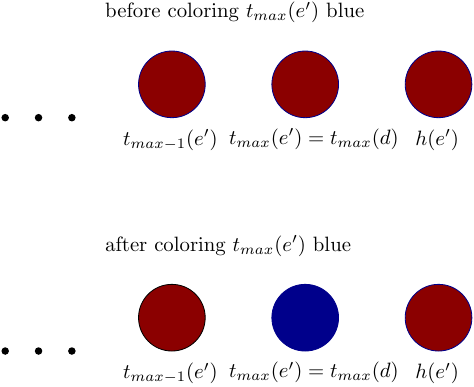}
			\caption{}
			\label{no_blue_2}
			
		\end{figure}
	
		We know that $t_{max}(e')$ and $h(e')$ are consecutive vertices in this order and $t_{max}(e')=t_{max}(d)$.
		Since $t_{max}(d)$ and $h(d)$ are not neighbors, $h(d)\neq h(e')$. When we color the vertex $t_{max}(e')$ blue, the vertices of hyperedge $e'$ with index less than $t_{max}(e')$ are red and the vertices of hyperedge $d$ with index less than $t_{max}(d)=t_{max}(e')$ are blue, so the hyperedges $e'$ and $d$ have exactly one common vertex, which is a tail-vertex of both and this is a contradiction. We proved that at the end of the coloring algorithm, there cannot be monochromatic blue hyperedges.}\end{proof}

	We showed that there is no monochromatic hyperedge in the resulting coloring, so we get a proper 2-coloring of $H$.

\section{A sufficient condition for 3-colorability of directed hypergraphs}

The following theorem gives a sufficient condition for proper 3-colorability from which Theorem \ref{R_{4}} follows directly. Already on five vertices there exists a directed hypergraph which satisfies the condition of Theorem \ref{general_head_tail} and its chromatic number is three. An example is the hypergraph $R$ in the Figure \ref{R}.

\begin{thm}\label{general_head_tail}

	Let $H$ be a directed hypergraph such that in every hyperedge the number of tail-vertices and the number of head-vertices are both at least one. Suppose that if $e_{1},e_{2}\in E(H)$ with $|e_{1}\cap e_{2}|=1$, then the common vertex is either a head-vertex of both hyperedges or a tail-vertex of both. Then $H$ admits a proper 3-coloring.
\end{thm}
\begin{proof}{
 Let $v_{1},v_{2},...,v_{n}$ be an arbitrary ordering of the vertex set of $H$. For a hyperedge $e\in E(H)$, denote the head-vertex of the hyperedge $e$ with the smallest index by $h_{min}(e)$, the head-vertex with largest index by $h_{max}(e)$, the tail-vertex with smallest index by $t_{min}(e)$ and the tail-vertex with largest index by $t_{max}(e)$. The set of the hyperedges can be divided into two parts according to whether the vertex of the hyperedge with largest index is a head-vertex or a tail-vertex of the hyperedge. Let denote by $E_{h}$ the set of the hyperedges such that its vertex with largest index is the head-vertex of the hyperedge and by $E_{t}$ the set of the hyperedges such that its vertex with largest index is a tail-vertex of the hyperedge.
	Take the colors blue, red and green. Let's start by making all the vertices blue. The coloring algorithm is the following.\newline 
	Step 1: Starting from the vertex with the smallest index, for each vertex $v_{i}$ we check whether there exists a monochromatic blue hyperedge in $E_{t}$ such that its vertex with the largest index is $v_{i}$. If no such hyperedge exists, we simply move on to the vertex $v_{i+1}$. If there is such a hyperedge, we color the vertex $v_{i}$ red and then move on to the vertex $v_{i+1}$.\newline
	Step 2: Starting from the vertex $v_{1}$, for each vertex $v_{j}$, we check whether there exists a monochromatic blue hyperedge $e\in E_{h}$ such that $h_{max}(e)=v_{j}$. If not, we proceed to the vertex $v_{j+1}$. If yes, we color the vertex $v_{j}$ green and then move on to the vertex $v_{j+1}$.\newline
	After Step 1, all hyperedges in $E_{t}$ will have at least one red vertex, and after Step 2, all hyperedges in $E_{h}$ will have at least one green vertex. It follows that after Step 2 there is no monochromatic blue hyperedge. We need to check that there are no monochromatic red and no monochromatic green hyperedges at the end of the coloring.\newline
	\clm{
		There is no monochromatic red hyperedge at the end of the coloring.}\newline
	\prf{
		It is enough to show that after Step 1, there is no monochromatic red hyperedge, because in Step 2 only the color green is used. Suppose for a contradiction that there is a monochromatic red hyperedge after Step 1, denote this hyperedge by $e$. We colored the vertex $h_{min}(e)$ red, hence when we checked the vertex $h_{min}(e)$ there was a hyperedge $f\in E_{t}$, which was monochromatic blue at this moment and $t_{max}(f)=h_{min}(e)$. Since $f\in E_{t}$, the vertex of $f$ with the largest index is $t_{max}(f)$. At the end of the coloring $e$ is monochromatic red, hence before the checking of vertex $h_{min}(e)$, all vertices of hyperedge $e$ with an index less than index of $h_{min}(e)$ have already been colored red. Every vertex of the hyperedge $f$ is blue before checking of the vertex $h_{min}(e)$ and the vertex of $f$ with the largest index is $h_{min}(e)$, since $f\in E_{t}$ and $t_{max}(f)=h_{min}(e)$. It follows that the only common vertex of $e$ and $f$ is $h_{min}(e)$, which is a head-vertex of hyperedge $e$ and a tail-vertex of hyperedge $f$, which is a contradiction. We showed that at the end of the coloring, there is no monochromatic red hyperedge.$\Box$}\newline
	\clm{
		At the end of the coloring there is no monochromatic green hyperedge.}\newline
	\prf{ 
	It is enough to show that Step 2 does not create a monochromatic green hyperedge. Suppose for a contradiction that there exists a monochromatic green hyperedge at the end of the coloring, denoted by $e$. Since the vertex $t_{min}(e)$ is colored green in Step 2, hence there is a hyperedge $f\in E_{h}$ which was monochromatic blue before the checking of vertex $t_{min}(e)$ and its vertex with largest index, namely $h_{max}(f)$ is equal to $t_{min}(e)$. The hyperedge $e$ can be monochromatic green at the end of the coloring if and only if all the vertices of hyperedge $e$ with an index less than index of $t_{min}(e)$ are colored green before checking of the vertex $t_{min}(e)$. The vertex of $f$ with the largest index is $t_{min}(e)$ and the vertex $t_{min}(e)$ was blue before checking of the vertex $t_{min}(e)$. It follows that the only common vertex of edges $e$ and $f$ is $t_{min}(e)$, which is a tail-vertex of $e$ and a head-vertex of $f$, which is a contradiction. There cannot be a monochromatic green hyperedge at the end of the coloring.$\Box$}

	We proved that at the end of the coloring there is not monochromatic hyperedge, so we get a proper 3-coloring of the hypergraph $H$. 
}\end{proof}

\section{Coloring 2 $\rightarrow$ 1 hypergraphs}

In the following, we prove the results for the chromatic number of $2\rightarrow 1$ hypergraphs which avoid one of the two-edge $2\rightarrow 1$ hypergraphs studied by Cameron.

\subsection{Avoiding hypergraphs $H_{2}$, $I_{1}$, $R_{3}$ and $E$}

\begin{proof}[Proof of Claim \ref{H_{2}}]

{We prove that for every integer $k\geq 2$, there exists a $2\rightarrow 1$ hypergraph $H=(V,E)$ such that it does not contain $H_{2}$ as a subhypegraph and the chromatic number of $H$ is at least k. We prove by induction. If $k=2$, then the statement is true, since the hypergraph with three vertices and one hyperedge requires at least 2 colors for a proper 2-coloring. Suppose that for $k$ the statement is true and that $A$ and $B$ are not necessarily different $2\rightarrow 1$ hypergraphs with chromatic number at least k. Consider the hypergraph $H$ for which $V(H)=V(A)\cup V(B)\cup\{x\}$ and $E(H)=E(A)\cup E(B) \cup \{ v_{1} v_{2}\rightarrow x: v_{1}\in V(A), v_{2}\in V(B)\}$. First we check that $H$ does not contain $H_{2}$ as a subhypegraph. Suppose for a contradiction that there exist two different hyperedges $e,f\in E(H)$ for which $e\cap f=\{u,v\}$ and $t(e)=t(f)=\{u,v\}$. It follows from the inductive assumption that $x$ is the head-vertex of at least one of them. Then $x$ is the head-vertex of the other hyperedge also, otherwise there could be at most one common vertex of $e$ and $f$. So $x$ is the head-vertex of both hyperedges and their tail-vertices are the same, which is a contradiction.
	To color the hypergraph $H$, we need at least $k+1$ colors, because if we use only $k$ colors, then using the inductive assumption, for any color both $A$ and $B$ contain a vertex with the given color, hence we can choose a vertex of the same color as $x$ from both $A$ and $B$. Then these two vertices and $x$ would create a monochromatic hyperedge. So $H$ does not contain $H_{2}$ as a subhypegraph and its chromatic number is at least $k+1$ finishing the proof.}
\end{proof}

\begin{proof}[Proof of Claim \ref{I_{1},R_{3},E}]
{ We prove that for every integer $k\geq 2$ there exists a $2\rightarrow 1$ hypergraph $H$ with chromatic number at least k and not containing any of the hypergraphs $I_{1}$, $R_{3}$ and $E$ as subhypergraphs. We prove by induction. If $k=2$, then for example the $2\rightarrow 1$ hypergraph with one hyperedge satisfies the conditions. Suppose that the statement is true for $k$. Using the inductive assumption, consider two $2\rightarrow 1$ hypergraphs with the same number of vertices which do not contain any of the hypergraphs $I_{1}$, $R_{3}$ and $E$ as subhypergraphs and have chromatic number at least $k$. Denote the two hypergraphs by $A$ and $B$, their vertex sets by $V(A)=\{a_{1},a_{2},...,a_{n}\}$ and $V(B)=\{b_{1},b_{2},...,b_{n}\}$. Let $H$ be the $2\rightarrow 1$ hypergraph for which $V(H)=V(A)\cup V(B) \cup \{x_{\sigma}: \sigma \in S_{n}\}$ and $E(H)=E(A)\cup E(B) \cup \{a_{i}b_{\sigma(i)}\rightarrow x_{\sigma}: i\in \{1,2,. ...,n\}, \sigma \in S_{n}\}$, where $S_{n}$ is the set of all permutations of $\{1,2,...,n\}$. \newline Let $e_{1},e_{2}\in E(H)$, for which $|e_{1}\cap e_{2}|=2$. If $e_{1},e_{2}\in E(A)\cup E(B)$, then it follows from the condition that the two common vertex of $e_{1}$ and $e_{2}$ are the tail-vertices of both hyperedges. If $e_{1}\in E(A)\cup E(B)$ and $e_{2}\in E(H)\setminus (E(A)\cup E(B))$, then $|e_{1}\cap e_{2}|\leq 1$, which would be a contradiction. Thus we can assume that $e_{1},e_{2}\in E(H)\setminus (E(A)\cup E(B))$. If $h(e_{1})=h(e_{2})$, then $e_{1}$ and $e_{2}$ must have only one common vertex, their head-vertex, which contradicts the fact that $|e_{1}\cap e_{2}|=2$. So the head-vertex of $e_{1}$ and the head-vertex of $e_{2}$ are different vertices and $h(e_{1}),h(e_{2})\in V(H)\setminus (V(A)\cup V(B))$, $t(e_{1}),t(e_{2})\subseteq (V(A)\cup V(B))$, hence $t(e_{1})=t(e_{2})$.
	It follows that $H$ does not contain any of the hypergraphs $I_{1}$, $R_{3}$ and $E$ as subhypergraphs. \newline We will prove that the chromatic number of $H$ is at least $k+1$. Suppose for a contradiction that there exists a proper $k$-coloring of $H$ and take such a coloring. Since the chromatic numbers of $A$ and $B$ are also at least $k$, we can choose from $V(A)$ and $V(B)$ $k$ vertices, which have different colors for each pair, let these vertices be $a_{i_{1}},a_{i_{2}},. ...,a_{i_{k}}\in V(A)$ and $b_{j_{1}},b_{j_{2}},...,b_{j_{k}}\in V(B)$. We can assume that $a_{i_{r}}$ and $b_{j_{r}}$ have the same color for all $1\leq r\leq k$. Let $\sigma\in S_{n}$ be a permutation such that $\sigma(i_{r})=j_{r}$ for all $1\leq r \leq k$. Consider the number $r'\in\{1,2,...,k\}$ for which the vertex $x_{\sigma}$ and the vertices $a_{i_{r'}}$, $b_{j_{r'}}$ have the same color. Then $a_{i_{r'}}b_{j_{r'}}\rightarrow x_{\sigma}\in E(H)$ is a monochromatic hyperedge, which is a contradiction. Hence the chromatic number of $H$ is at least $k+1$.
}
\end{proof}

\subsection{Avoiding the hypergraphs $I_{0}$ and $R_{4}$}

We prove that avoiding the hypergraph $I_{0}$ is a sufficient condition for proper 4-colorability and that there exists a $2\rightarrow 1$ hypergraph that avoids $I_{0}$ and has a chromatic number three, an example is the following hypergraph $I$. The hypergraph $I$ is illustrated in the Figure \ref{I}, where each row corresponds to a hyperedge, the black points represent the head-vertices of the hyperedges and the white points represent the tail-vertices of the hyperedges.
\begin{eqnarray*}
	V(I)&=&\{v_{1},v_{2},v_{3},v_{4},v_{5}\}\\ E(I)&=&\{v_{1}v_{2}\rightarrow v_{3},
	v_{2}v_{3}\rightarrow v_{4},
	v_{3}v_{4}\rightarrow v_{5},
	v_{4}v_{5}\rightarrow v_{1},
	v_{1}v_{5}\rightarrow v_{2},\\
	& & v_{1}v_{3}\rightarrow v_{4},
	v_{2}v_{4}\rightarrow v_{5},
	v_{3}v_{5}\rightarrow v_{1},
	v_{1}v_{4}\rightarrow v_{2},
	v_{2}v_{5}\rightarrow v_{3}\}	
\end{eqnarray*}

\begin{figure}[h!]
	\centering
	\includegraphics[width=0.3\textwidth]{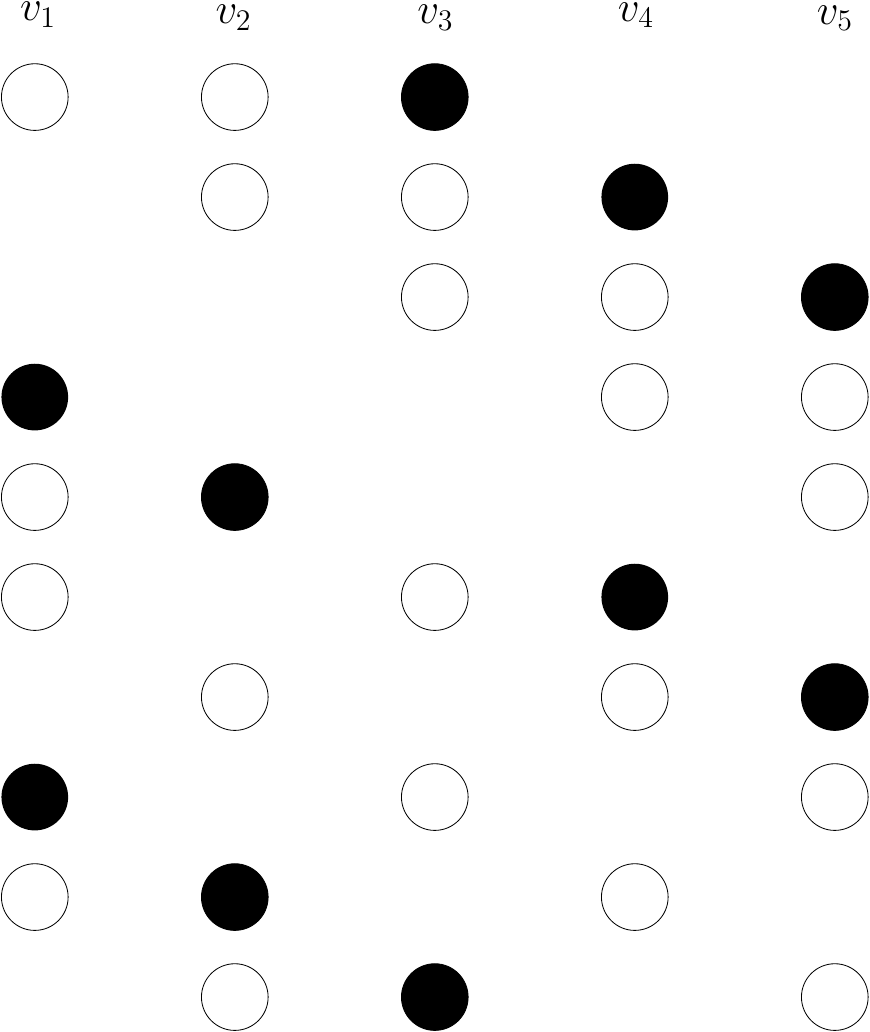}
	\caption{The 3-chromatic hypergraph $I$, which avoids $I_{0}$}
	\label{I}
	
\end{figure}

It is easy to check that $I$ avoids the hypergraph $I_{0}$ and that its chromatic number is three. No matter how the five vertices of the hypergraph $I$ are colored by two colors, there are always three vertices of the same color because of the pigeonhole principle and since any three vertices form a hyperedge, it is a monochromatic hyperedge. However for three colors it is possible to give a coloring in which there are no three vertices of the same color, which gives a proper 3-coloring of $I$. Hence the chromatic number of $I$ is three.

\begin{proof}[Proof of Theorem \ref{I_{0}}]
{
It is enough to consider the case when there are not two hyperedges with the same vertex set. For a given vertex $u$, denote by $E_{h}(u)$ the set of hyperedges whose head-vertex is $u$.
	
	\clm{
	For every $u\in V$, exactly one of the following holds:\newline
	$(a)$ $E_{h}(u)$ is the empty set. \newline
	$(b)$ There exists $v\in V\setminus\{u\}$ such that $v\in e$ for every $e\in E_{h}(u)$.\newline
	$(c)$ $E_{h}(u)=\{vw\rightarrow u, wz\rightarrow u, vz\rightarrow u\}$ for some vertices $v,w,z\in V\setminus\{u\}$.}
	
	\prf{ 
	If $|E_{h}(u)|\leq 1$, then it is easy to check that the statement is true. If $|E_{h}(u)|=2$, then the two hyperedges have a common vertex different from $u$, because otherwise the two hyperedges would have only one common vertex, which is the head-vertex of both, so the two hyperedges would create $I_{0}$. Hence in the case $|E_{h}(u)|=2$ $(b)$ is satisfied. \newline We can assume that $|E_{h}(u)|\geq 3$ and let $e_{1},e_{2},e_{3}$ be three hyperedges with head-vertex $u$. Since $h(e_{1})=h(e_{2})$, there exists a vertex $v\in V$ which is a vertex of both hyperedges. Similarly, for $e_{2}$ and $e_{3}$, there exists a vertex $w\in V$ which is a vertex of both hyperedges. We will prove that if $v$ and $w$ are the same, then $(b)$ is satisfied, and if $v$ and $w$ are different, then $(c)$ is satisfied.\newline 
	Suppose that $v$ and $w$ are equal. Then $v$ is a vertex of each of the hyperedges $e_{1}, e_{2}$ and $e_{3}$. If $|E_{h}(u)|=3$, we are done, since $(b)$ is satisfied. If $|E_{h}(u)|\geq4$, then take an arbitrary hyperedge $e_{4}$ such that $e_{4}\in E_{h}(u)\setminus \{e_{1},e_{2},e_{3}\}$. Since $h(e_{4})=h(e_{i})$ is satisfied for $i=1,2,3$, the hyperedge $e_{4}$ has a common vertex with each of the hyperedges $e_{1}, e_{2}$ and $e_{3}$ which is different from $u$. The third vertex of hyperedges $e_{1},e_{2}$ and $e_{3}$ different from $u$ and $v$ are pairwise different, so $|(e_{1}\cup e_{2} \cup e_{3})\setminus \{u,v\}|=3$. It follows that $e_{4}$ can have a common vertex different from $u$ with each of the hyperedges $e_{1},e_{2}$ and $e_{3}$ if and only if $v\in e_{4}$. Thus we proved that for every $e\in E_{h}(u)$ $v\in e$ holds. \newline
	Suppose that the vertex $v$ and vertex $w$ are different. Then the two tail-vertices of $e_{2}$ are $v$ and $w$. Let $z$ denote the third vertex of the hyperedge $e_{1}$ different from $u$ and $v$. Since $h(e_{1})=h(e_{3})$, the two hyperedges have a common vertex different from $u$. One of the tail-vertices of $e_{3}$ is $w$. If $z=w$, then $e_{1}$ and $e_{2}$ are equal. If $z$ is the common vertex of $e_{1}$ and $e_{3}$ which is different from $u$, then $\{e_{1},e_{2},e_{3}\}=\{vz\rightarrow u, vw\rightarrow u, wz\rightarrow u\}$. If $|E_{h}(u)|=3$, then we proved that $(c)$ is satisfied. Suppose for a contradiction that $|E_{h}(u)|\geq4$ and let $f\in E_{h}(u)\setminus \{e_{1},e_{2},e_{3}\}$. Since $u$ is the head-vertex of hyperedges $e_{1},e_{2},e_{3}$ and $f$, the hyperedge $f$ has a common vertex different from $u$ with each of the hyperedges $e_{1},e_{2}$ and $e_{3}$. Hence $f$ contains at least two of the vertices $v,w$ and $z$, but then $f$ is equal to one of the hyperedges $e_{1}, e_{2}$ and $e_{3}$, which is a contradiction. So we have seen that if $v$ and $w$ are different, then the case $(c)$ is satisfied. $\Box$
	}\newline
	
	If $E_{h}(u)$ is the empty set, then take an arbitrary vertex $v$ different from $u$ and add to the hypergraph the hyperedge $wv\rightarrow u$ for all $w\in V\setminus\{u,v\}$. If for $u$ the case $(b)$ is satisfied, then there exists $v\in V$ for which $v\in e$ for all $e\in E_{h}(u)$. Then, add to the hypergraph all possible hyperedges such that its head-vertex $u$ and $v$ is one of its tail-vertices. It can be easily checked that the resulting hypergraph still avoids $I_{0}$, so the following statement holds. \newline
	
	\clm{
	We can assume that for every $u\in V$ exactly one of the following holds:\newline
	$(a)$ $E_{h}(u)=\{vw\rightarrow u, wz\rightarrow u, vz\rightarrow u\}$ for some vertices $v,w,z\in V\setminus\{u\}$.\newline
$(b)$ $E_{h}(u)=\{vw\rightarrow u: w\in V\setminus \{u,v\}\}$ for some vertex $v\in V\setminus\{u\}$.}\newline
	
	We continue with the proof of the theorem.
	Denote by $V_{a}$ the vertices $u\in V$ for which $E_{h}(u)=\{vw\rightarrow u, wz\rightarrow u, vz\rightarrow u\}$ for corresponding vertices $v,w,z\in V\setminus\{u\}$ and by $V_{b}$ the vertices $u$ for which $E_{h}(u)=\{vw\rightarrow u: w\in V\setminus \{u,v\}\}$ for a corresponding vertex $v\in V$.
Let $v_{1},v_{2},...,v_{n}$ be an arbitrary ordering of elements of the vertex set $V$. For a hyperedge $e\in E(H)$, let $h(e)$ be the head-vertex of $e$, $t_{1}(e)$ the tail-vertex of $e$ with smaller index and $t_{2}(e)$ the tail-vertex of $e$ with larger index. The set of the hyperedges can be divided into three parts with respect to the order of their head-vertex and tail-vertices. Denote those hyperedges such that its head-vertex is the vertex of the hyperedge with i-th  smallest index by $E_{i}$. \newline Take the colors blue, red, green and yellow. Let all vertices initially be blue. In the following, we give a proper 4-coloring.\newline
	Step 1: Starting from the vertex with the smallest index, for each vertex $v_{i}$ we check if there exists a hyperedge from $E_{3}$ such that its head-vertex is $v_{i}$ and it is monochromatic blue. If no such hyperedge exists, we simply move on to the next vertex. If so, we color the vertex $v_{i}$ red and move on to the vertex $v_{i+1}$.\newline
	Step 2: Starting from the vertex with the largest index and going backwards for each vertex $v_{i}$, we check if there is a hyperedge from $E_{1}$ such that has no green vertex. If no such hyperedge exists, we move on to the vertex $v_{i-1}$. If there is such a hyperedge, we color vertex $v_{i}$ green and move on to the next vertex.\newline
	Step 3: Starting from the vertex with the smallest index, for each vertex $v_{i}$, we check whether there exists a monochromatic blue hyperedge from $E_{2}$ such that its head-vertex is $v_{i}$. If not, we move on to the vertex $v_{i+1}$. If such a hyperedge exists, then we color the vertex $v_{i}$ yellow and we move on to the next vertex.\newline
	
	\clm{
		At the end of Step 1, there is no monochromatic red hyperedge in $E$ and there is no monochromatic blue hyperedge in $E_{3}$.
	}\newline
	\prf{ 
	It is easy to check that in step 1, at least one vertex of each monochromatic blue hyperedge in $E_{3}$ is colored red, so there cannot be a monochromatic blue hyperedge in $E_{3}$ at the end of step 1. Suppose for a contradiction that there exists a monochromatic red hyperedge at the end of step 1, denote such a hyperedge by $e$. Since $e$ is monochromatic red at the end of step 1, 
	we have already colored the vertices of hyperedge $e$ with index less than the index of $h(e)$ red before checking the head-vertex of $e$. The vertex $h(e)$ is colored red, so there exists a hyperedge $f\in E_{3}$ such that its head-vertex is $h(e)$ and it was monochromatic blue before checking the vertex $h(e)$. Since $f$ has no vertex with index greater than the index of $h(e)$, and all vertices of $f$ are blue and all vertices of $e$ with index less than the index of $h(e)$ are red before checking of vertex $h(e)$, hence the only common vertex of $e$ and $f$ is $h(e)$, which is the head-vertex of both hyperedges, and this contradicts the fact that $H$ does not contain $I_{0}$ as a subhypergraph. So there is no monochromatic red hyperedge at the end of step 1. $\Box$.
	}\newline
	
	\clm{
	At the end of Step 2, there cannot be monochromatic hyperedges with color red or green in $E$ and there are no monochromatic blue hyperedges in $E_{1}\cup E_{3}$.}
	
	\prf{ 
	It is easy to check that a hyperedge from $E_{1}\cup E_{3}$ cannot be monochromatic blue, since at the end of step 1 all hyperedges from $E_{3}$ have at least one red vertex, and in step 2 at least one vertex of all hyperedges from $E_{1}$ are colored green. We know that at the end of step 1 there is not monochromatic red hyperedge and in step 2 we use only the color green, so monochromatic red hyperedges can not be created. It is enough to show that there is not monochromatic green hyperedge at the end of step 2. The proof is analogous to the proof of the previous claim. $\Box$}\newline
	
	\clm{
	 There is no monochromatic hyperedge at the end of the Step 3.}
	
	\prf{ 
	It follows from the previous claim that if a hyperedge is monochromatic before step 3, it can only be blue and be in $E_{2}$. By the definition of step 3, it follows that every such hyperedge will have a yellow vertex. Hence it is enough to show that no monochromatic yellow hyperedge is created in step 3. Since only vertices with color blue are colored yellow and every hyperedge from $E_{1}\cap E_{3}$ has a vertex with a color different from color blue at the end of step 2, a monochromatic yellow hyperedge from $E_{1}\cap E_{3}$ can not be created. We are left to prove that a hyperedge in $E_{2}$ cannot be monochromatic yellow at the end of the coloring either.
	Suppose for a contradiction that there exists a monochromatic yellow hyperedge from $E_{2}$ and let $e$ be such a hyperedge. The hyperedge $e$ can be monochromatic yellow at the end of the coloring if and only if $t_{1}(e)$ is colored yellow before checking the vertex $h(e)$. Since the vertex $h(e)$ is colored yellow, there exists a hyperedge $f\in E_{2}$ such that $h(f)=h(e)$ and all vertices of $f$ are blue before checking the vertex $h(e)$. Since the head-vertex of $e$ and the head-vertex of $f$ is the same, there is an other common vertex different from $h(e)$, which can only be $t_{2}(e)$, because $t_{1}(e)$ is yellow and $t_{1}(f)$ is blue before checking the vertex $h(e)$. Thus, the tail-vertices of $e$ and $f$ with higher index are the same. There are two cases, $h(e)\in V_{a}$ or $h(e)\in V_{b}$. If $h(e)\in V_{a}$, then $g=[t_{1}(f)t_{1}(e)\rightarrow h(e)]\in E(H)$. We know that $t_{1}(e)$ is yellow, $t_{1}(f)$ and $h(e)$ are blue before checking the vertex $h(e)$, which implies that $g$ was monochromatic blue before step 3. The two tail-vertices of $g$, $t_{1}(e)$ and $t_{1}(f)$, have smaller indices than index of vertex $h(e)$, which is the head-vertex of $g$, so $g\in E_{3}$. Hence $g$ is a hyperedge in $E_{3}$, which was monochromatic blue before step 3, this is a contradiction.
		Thus we can assume that $h(e)\in V_{b}$. Since $v_{n}$ cannot be the head-vertex of a hyperedge from $E_{2}$, the color of $v_{n}$ will never be yellow, which implies that $t_{2}(e)$ is a vertex different from $v_{n}$. The hyperedge $g'=[t_{2}(e)v_{n}\rightarrow h(e)]\in E(H)$, because $h(e)\in V_{b}$. The vertex $v_{n}$ cannot be the head-vertex of a hyperedge from $E_{1}$, hence it cannot be green at the end of step 2. We know that both $t_{2}(f)$ and $h(e)$ were blue at the end of step 2. This implies that $g'$ is a hyperedge in $E_{1}$ such that it does not have a vertex with color green at the end of step 2, which is a contradiction. We showed that there is not monochromatic yellow hyperedge at the end of the coloring. $\Box$
	}\newline
	This shows that the given coloring is a proper 4-coloring of $H$.}
\end{proof}

The following example shows that there is a $2\rightarrow 1$ hypergraph satisfying the condition of Theorem \ref{R_{4}} and has a chromatic number three.

\begin{eqnarray*}
	V(R)&=&\{v_{1},v_{2},v_{3},v_{4},v_{5}\}\\ E(R)&=&\{
	v_{2}v_{3}\rightarrow v_{1},
	v_{2}v_{4}\rightarrow v_{3},
	v_{3}v_{4}\rightarrow v_{5},
	v_{4}v_{5}\rightarrow v_{1},
	v_{1}v_{2}\rightarrow v_{5},\\
	& & v_{1}v_{4}\rightarrow v_{2},
	v_{3}v_{5}\rightarrow v_{2},
	v_{1}v_{3}\rightarrow v_{4},
	v_{2}v_{5}\rightarrow v_{4},
	v_{1}v_{5}\rightarrow v_{3}\}	
\end{eqnarray*}
The hypergraph $R$ is illustrated in \ref{R}, each row represents a hyperedge, the black points are the head-vertices of the hyperedges, and the white points are the tail-vertices of the hyperedges.
\begin{figure}[h!]
	\centering
	\includegraphics[width=0.3\textwidth]{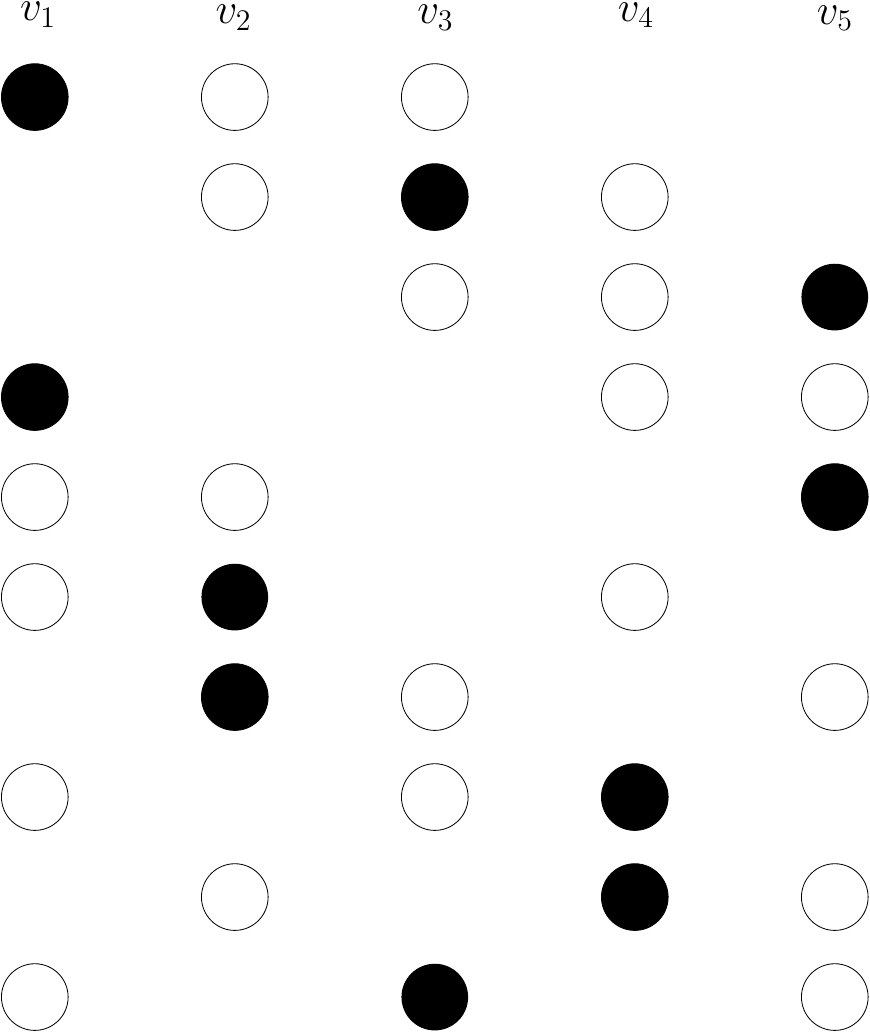}
	\caption{The 3-chromatic hypergraph $R$, which avoids $R_{4}$}
	\label{R}
	
\end{figure}
It is easy to check that the hypergraph $R$ avoids the hypergraph $R_{4}$ as a subhypergraph and its chromatic number is three.

We prove that avoiding hypergraphs $I_{0}$ and $R_{4}$ is a sufficient condition for proper 2-colorability.

\begin{proof}[Proof of Theorem \ref{I_{0},R_{4}}]
	{
	Take the red and blue colors and let $v_{1},v_{2},...,v_{n}$ be an arbitrary ordering of vertices of $H$. From now on let the head-vertex of a hyperedge $e$ be denoted by $h(e)$, its tail-vertex with smaller index by $t_{1}(e)$, its tail-vertex with larger index by $t_{2}(e)$.
	\newline
 Let the color of all vertices be blue.
	Starting from the vertex $v_{1}$, for each vertex $v_{i}$, we check whether there exists a hyperedge such that its head-vertex is $v_{i}$ and the hyperedge is monochromatic blue. If no such hyperedge exists, we move on to the vertex $v_{i+1}$. If there is such a hyperedge, then we color the vertex $v_{i}$ red and then we move on to the vertex $v_{i+1}$. Note that in the resulting coloring, all hyperedges will have red vertices, so no hyperedge can be monochromatic blue.
	 It is enough to show that there is no monochromatic red hyperedge. Suppose for a contradiction that there exists a hyperedge $e$ which is monochromatic red at the end of the coloring. Since each vertex is colored red at most once, the only way to make the hyperedge $e$ monochromatic red is to color all vertices red. So, when we color the vertex of the hyperedge $e$ with the largest index red, the other two vertices of the hyperedge $e$ are already colored red. Let $v$ denote the vertex of the hyperedge $e$ with the largest index.
	 Since the vertex $v$ was colored red, there was a hyperedge $f\in E(H)$ such that $h(f)=v$ and $f$ was monochromatic blue before checking the vertex $v$ . Assume first that $v$ is the head-vertex of $e$, then the only common vertex of $e$ and $f$ is $v$, since before checking the vertex $v$, the other two vertices of the hyperedge $e$ are colored red, the other two vertices of the hyperedge $f$ are colored blue. Hence the hyperedges $e$ and $f$ have only one common vertex, which is the head-vertex of both hyperedges, and this is a contradiction. Thus we can assume that $v$ is a tail-vertex of $e$. Then $v$ is also their only common vertex, which is the tail-vertex of one hyperedge and the head-vertex of the other hyperedge, which is also a contradiction. So there is no monochromatic red hyperedge. We showed that the coloring is a proper 2-coloring of $H$.
}\end{proof}

\section{Open questions}
While we proved that Conjecture \ref{conjecture} is true for directed hypergraphs with all hyperedges having size at least three and exactly one head-vertex, the conjecture is still open in the case a hyperedge can have more head-vertices. We know that avoiding the hypergraph $I_{0}$ is a sufficient condition for proper $4$-colorablity and that there exists a $2\rightarrow 1$ which avoids $I_{0}$ and its chromatic number is three. It is open whether avoiding the hypergraph $I_{0}$ guarantees proper $3$-colorability.

\section{Acknowledgment}
I am grateful to Balázs Keszegh for the valuable discussions, helpful remarks and for reading this manuscript.

\end{document}